\documentclass[a4paper,12pt]{article}
\usepackage{amsmath}
\usepackage{amssymb}
\usepackage{amsthm}
\usepackage[latin1]{inputenc}
\usepackage[OT1]{fontenc}
\usepackage{mathrsfs}
\usepackage{graphicx}

\newtheorem{theorem}{Theorem}[section]
\newtheorem{lemma}[theorem]{Lemma}
\newtheorem{proposition}[theorem]{Proposition}
\newtheorem{corollary}[theorem]{Corollary}

\newcommand{\proba}{\mathbb{P}}

\newcommand{\EE}{\mathbb{E}}
\newcommand{\esperance}{\mathbb{E}}
\newcommand{\PP}{\mathbb{P}}
\newcommand{\NN}{\mathbb{N}}

\newcommand{\RR}{\mathbb{R}}

\newcommand{\B}{\cal B}

\newcommand{\moche}{\mathcal{S}}
\newcommand{\tildeb}{{\tilde b}}
\newcommand{\ty}{{\tilde y}}
\begin{document}
\title{Stochastic models for a chemostat and long time behavior}
\author{P. Collet\thanks
{ Centre de Physique Th\'eorique, CNRS UMR 7644,
 Ecole Polytechnique
 F-91128 Palaiseau Cedex (France);
 e-mail: \texttt{collet@cpht.polytechnique.fr}}, Servet Mart\'inez\thanks{Departamento Ingenier{\'\i}a Matem\'atica and Centro
Modelamiento Matem\'atico, Universidad de Chile,
UMI 2807 CNRS, Casilla 170-3, Correo 3, Santiago, Chile;
\texttt{e-mail: smartine@dim.uchile.cl}},  Sylvie
  M\'el\'eard\thanks{CMAP, Ecole Polytechnique, CNRS, route de Saclay, 91128 Palaiseau
    Cedex-France; e-mail: \texttt{sylvie.meleard@polytechnique.edu}}, Jaime San Mart\'in\thanks{Departamento Ingenier{\'\i}a Matem\'atica and Centro
Modelamiento Matem\'atico, Universidad de Chile,
UMI 2807 CNRS, Casilla 170-3, Correo 3, Santiago, Chile;
\texttt{e-mail: jsanmart@dim.uchile.cl
}}}

\bigskip

\maketitle

\begin{abstract}
We introduce two  stochastic chemostat models consisting in a coupled population-nutrient  process  reflecting the interaction between the nutrient and the bacterias in the chemostat with finite volume. The nutrient concentration evolves continuously but depending on the population size, while the population size is a  birth and death process with coefficients depending on time through the nutrient concentration. The nutrient is shared by the bacteria and creates  a regulation of the bacterial population size. The latter and the fluctuations due to the random  births and deaths of individuals  make the population go almost surely to extinction. Therefore,   we are interested in the long time behavior of the bacterial population  conditioned to the non-extinction.  We prove the global existence of the process and  its almost sure extinction. The existence of quasi-stationary distributions is obtained  based on a general fixed point argument. Moreover, we prove the absolute continuity of the nutrient distribution when conditioned to a fixed number of individuals and the smoothness of the corresponding densities. 
\end{abstract}

\parindent=0pt

\section{Introduction}

Since some decades and the first work of J. Monod  \cite{monod} and Novik and Szilar  \cite{ns}, \cite{ns2}, see also \cite{SW}, biologists have developed procedures which allow to maintain a bacterial population at a stationary finite size while at the same time, the bacteria have a constant individual growth rate. The procedure is based on a chemostat: bacteria live in a growth container of constant volume in which liquid is injected continuously. 
The entering  liquid  contains  a fixed concentration of nutrient but no bacteria (fresh liquid). In the container, the nutrient is consumed by the bacteria.  We assume that the chemostat is well stirred, so that the distribution of bacteria and nutrient are spatially uniform.   Since the container has a finite volume and fresh liquid enters permanently, an equal amount of liquid pours out containing both unconsumed nutrients and bacteria. 

These chemostats are extremely useful in the study of bacterial population dynamics, in particular  in the study of  the  selection of the fastest growing species or  the fixation of advantageous mutations (see \cite{bgm}, \cite{bbgm}, \cite{dh}  or \cite{lls}). In the literature, their study is mainly  based on deterministic models where both nutrient and bacteria population dynamics are described by a coupled deterministic continuous process. Deterministic approximations for the bacterial population's size are justified by a large population assumption. 

\medskip
In this work, we develop stochastic chemostat models based on a previous work of Crump and O' Young \cite{CW} taking into account that the bacteria population may not be large enough so that a deterministic approximation can be justified. We introduce a coupled population-nutrient process  reflecting the interaction between the nutrient and the bacterias in the chemostat. The nutrient concentration evolves continuously but depending on the population size, while the population size is a  birth and death process with coefficients depending on time through the nutrient concentration. Moreover, the time derivative of the nutrient concentration jumps simultaneously with the population size. 
 We thus take into account the random fluctuations of this population size due to the individual births and deaths.
 The bacteria need nutrient to reproduce. We will consider two cases. In the first one,  the bacterial population dies instantaneously if the nutrient is missing. In the second case bacteria can survive without nutrient by undergoing some kind of hibernation and may wake up once nutrient reappears. 
The nutrient is shared by the bacteria. This creates an indirect competition between bacteria and leads to a regulation of their population size. In our models, the fluctuations due to the random  births and deaths of individuals and the size regulation make the population go almost surely to extinction. Therefore, the long time behavior of the population's size is obvious  and  the interesting questions concern firstly the rate of extinction and secondly the long time behavior conditioned to the non-extinction  which is captured by the notion of quasi-stationary distribution.

\medskip
To our knowledge, the models introduced in this paper are the first stochastic chemostat models  where interaction between bacteria is taken into account leading to extinction. The study of quasi-stationarity gives nevertheless a description of a quasi-stability which can happen in a faster time scale than extinction. This work concerns monotype individuals but could be generalized to a multi-type population.

\medskip
In Section 2, we describe the two population-nutrient  models described above and prove in Section 3  their  global existence. We also show the extinction of the population when time increases. The existence of quasi-stationary distributions is obtained in Section 4. Our main theorem is based on a general argument proved in \cite{quatres}.  In Section 5, we prove the absolute continuity of the nutrient distribution when conditioned to a fixed number of individuals and the smoothness of the corresponding densities.

\section{The Stochastic  Population-Nutrient Process}

We consider a stochastic discrete population process describing the dynamics of a bacteria
population for which individuals develop and reproduce depending on the quantity of 
nutrient $y$ in the solution. The dynamics of the nutrient is related to the 
consumption of the individuals. We assume that the concentration of nutrient in the 
injected solution (without bacteria) is a constant equal to $y^*$.  The chemostat has a finite volume equal to one. The liquid enters in the chemostat free of bacteria and  pours out after being well stirred in the container. The pouring  out liquid contains bacteria. The dilution coefficient of nutrient in the fresh liquid per unit of time is $D$. Since the liquid is well stirred, around $N(t)D$ bacterias will be washed out in the pouring out, when $N(t)$ is the size of the bacteria population at time $t$. Thus $D$ is also  the rate at which an individual will disappear due to the evacuation of liquid.

We consider  coupled processes in which the nutrient concentration evolves continuously while the bacteria population size  
 evolves as a time-continuous birth and death process with coefficients depending on the nutrient concentration.  We assume that the nutrient is partly consumed during the reproduction of bacteria. 
 
 We will denote by $(Y(t), t\geq 0)$ the concentration of nutrient and by 
$(N(t): t\geq 0)$ the population size process. The stochastic process $Z=(Z(t):=(N(t), Y(t)): t\geq 0)$ describes 
both the population size and the nutrient concentration in the chemostat.  

\bigskip
Let us now define  the parameters of the model.

\medskip

If  $y$ is the 
quantity of nutrient, then the birth and death parameters driving the dynamics of the population are as follows.

\medskip

\noindent $\bullet$ The birth rate per individual is 
$b(y)$, where the function $b:\RR_+\to \RR_+$ is assumed to be 
an increasing continuous function and such that $b(0)=0$ and $b(y)>0$ 
for $y>0$. We assume that $b$ is bounded with  an upper-bound 
$b_{\infty} $. 

\medskip

A usual example of a function $b$ is, for some constant  $K>0$,
$$
b(y)=b_{\infty} \,{y\over K+y} \,.
$$
An extra hypothesis that we will add for some results is that 
$b$ is differentiable and ${db\over dy}(0)>0$.

\medskip

\noindent $\bullet$ The background death rate per individual is $d(y)$,
so it is supposed to be a function of the concentration of nutrient.  The function
$d:\RR_+\to \RR_+\cup \{\infty\}$ is assumed to be continuous
non-increasing, strictly positive, and $d(0)$ is the unique value that
can be infinite.  

\medskip

\noindent $\bullet$  The dilution makes each individual disappear at rate $D$ independently  of the birth and death events.

\medskip



\noindent $\bullet$  The per individual rate of consumption of nutrient for reproduction  
is $ {b(y)\over R}$, where $R$ denotes the biomass yield.
Furthermore,  individuals  consume nutrient during their life and  the quantity of nutrient  consumed per individual will be denoted by 
$\eta\geq 0$.

\noindent $\bullet$  We will consider two cases.

\noindent  In the first case,  individuals need nutrient to survive. Then we will assume  that their death is instantaneous as soon as nutrient is missing, therefore
$d(0)=\infty$. 

\noindent  In the second case, bacteria enter in some kind of hibernation if nutrient is missing. That means that $d_{0}(0)$ can be finite. 

\noindent  In  both  cases, we will set $$\tilde b(y) = {b(y)\over R} + \eta\, {\bf 1}_{y>0}.$$ 

\medskip

Let us now describe the process. In both cases the nutrient concentration process $\,Y=(Y(t): t\geq 0)$ evolves according to 
\begin{equation}
\label{frontera0}
{d Y(t)\over dt} = 
D(y^*-Y(t)) - \tilde b(Y(t))\, N(t) \,.
\end{equation}

The process $Z$
has the following infinitesimal generator: for $(n,y)\in \mathbb{N}\times\mathbb{R}_{+}$,
\begin{eqnarray}
\nonumber
\mathcal{L}f(n,y)&=&b(y)\,n\, f(n+1,y)+ (D+d(y))\,n\,f(n-1,y)\\
\nonumber
&& -(b(y)+D+d(y))\,n\,f(n,y) \\
\label{generator}
&& +\left(D(y^*-y)-n\,\tildeb(y)\right)\; \partial_{y}f(n,y)\;.
\end{eqnarray}

\medskip  We refer to  \cite{CJL}  for the numerical study of  similar models.

\bigskip
Below, we show that the hypotheses we gave on
the coefficients leading the process, guarantee that $Z$ is well-defined
and  takes values in $\mathbb{N}\times [0,y^*]$.

\begin{proposition} 
\label{defproc}
The process $Z$ is well defined and takes values in
$\NN\times \RR_+$ for all positive time $t\in \mathbb{R}_{+}$. 
Moreover, $\NN_+\times [0,y^*]$ is an invariant set for the process $Z$,
so $Y(0)\in [0,y^*]$ implies $Y(t)\in [0,y^*]$ for all $t\in \RR_+$.
\end{proposition}

\begin{proof} 
The process $N=(N(t): t\ge 0)$ has a pathwise representation driven by a Point 
Poisson measure ${\cal N}(d\theta, ds)$ defined on 
$\mathbb{R}_{+} \times \mathbb{R}_{+}$: 
\begin{eqnarray} 
\nonumber
N(t) &=& N_{0}\! + \!\!\!\!\! \int\limits_{\mathbb{R}_{+} \times (0,t]}\!\!\!\!\! 
{\bf 1}_{\theta\leq b(Y_{s})} {\cal N}(d\theta, ds) \!-\! 
\!\!\!\!\!\int\limits_{\mathbb{R}_{+} \times (0,t]} \!\!\!\!\! 
{\bf 1}_{b(Y_{s})\leq \theta\leq b(Y_{s})+ D + d(Y_{s})} 
{\cal N}(d\theta, ds);\\ 
\label{systeme}
Y(t) &=& Y(0) +\!\! \int_{0}^t \!\! \bigg(D(y^* - Y(s)) - \tildeb(Y(s)) 
\, N_{s}\bigg) ds.
\label{systeme} 
\end{eqnarray}

It is obvious that the process $N$ 
is stochastically upper-bounded by a birth 
process with individual birth rate $b_{\infty}$. This latter does not explode, so 
does $N$.

\medskip 

Let us now study the nutrient concentration $Y=(Y(t): t\geq 0)$. Note that
the assumptions on the parameters and Equation  (\ref{frontera0}) ensure  that $Y(t)\geq 0$.  Indeed, the derivative of $Y(t)$ 
at $y=0$ cannot be negative. 

\medskip

Let us show that $\NN\times [0,y^*]$ is invariant. Take
$Y(0)\in [0, y^*]$. A standard comparison theorem yields $Y(t)\leq v(t)$ where 
$v'(t) = D(y^*-v(t))\ ;\ v(0)=Y(0)$. 
But in that case, $v(t)= y^* - (y^*-Y(0))e^{-Dt} \leq y^*$ then it remains
in $ [0,y^*]$ forever proving the invariance.  

\end{proof}

For the initial condition $N(0)\in \NN^*$ 
and $Y(0)\in \mathbb{R_{+}}\setminus [0,y^*]$
we have that at time $t=(Y(0)-y^*)/\tildeb(y^*)$ the process has already 
attained the invariant set $[0, y^*]\times \RR_+$ or became extinct.

\section{Study of Extinction}

We are now interested in  studying the extinction of the population  or  the complete consumption of the nutrient or other specific states of the population-nutrient process and the associated hitting times. 

\medskip
Let $\B(\NN\times \RR_+)$ be the class of Borel sets of $\NN\times \RR_+$
and $\B(\RR_+)$ the class of Borel sets of $\RR_+$.
For $D\in \B(\NN\times \RR_+)$ we put by
$T^D=\inf\{t\ge 0: Z(t)\in D\}$ the hitting time of $D$ by the process, with 
the usual convention $\infty=\inf \emptyset$.   
The hitting time of the boundaries will be denoted in such a way that 
the reference to process $N$ will be avoided, the contrary for $Y$ where 
we will explicit it. More precisely,
\begin{equation}
\label{hittimes}
T_0\!=\!T^{\{0\}\times \RR_+}\!=\!\inf\{t\!\ge \!0: N(t)\!=\!0\},\; 
T_{Y=0}\!=\!T^{\NN\times \{0\}}\!=\!\inf\{t\!\ge \!0: Y(t)\!=\!0\}.
\end{equation}
We will also denote by $T_{\le m}=T^{\{0,..,m\}\times \RR_+}=\inf\{t\ge 0: N(t)\le m\}$.
Analogously for  $T_{< m}$.

\medskip 

Note that the set $\{0\}\times \RR_+$ is an absorbing set, that is $N(t)=0$ for all
$t\ge T_0$. 

\medskip

After $T_0$ the nutrient $Y(t)$ is absorbed linearly at $y^*$, in fact: 
$Y(t)=y^*$ for $t\ge T_0+(y^*-Y(T_0))/D\,y^*$ and $Y(t)=Y(T_0)+D(t-T_0)$ for 
$t\in [T_0,T_0+(y^*-Y(T_0))/D]$.

\medskip

Let us see what happens after $T_{Y=0}$.  In the case $\eta\ge D y^*$  
the nutrient $Y(t)$ remains at $0$ up to the extinction of the
population, so $Y(t)=0$ when $T_{Y=0}<t<T_0$. The evolution of
the nutrient after $T_{0}$ was already described. 
Now assume $n\,\eta >  Dy^*$ for some $n\ge 1$, and denote $n_0$ the
minimal of these values, so $(n_0-1)\,  \eta\leq Dy^* < n_0 \,\eta$.
In this case $Y(t)=0$ when $T_{Y=0}<t<T_{< n_0}$.  
In the case where $d_{0}(0) = \infty$, we get
$$T_{Y=0} \leq T_{0} \Longrightarrow  T_{Y=0}  = T_{0}$$ because all individuals die instantaneously. 


\bigskip
Let us firstly study the stationary nutrient concentration states at fixed  population size. 
\medskip

\begin{lemma} 
\label{roots1}
Consider the equation
\begin{equation}
\label{root} 
G_n(y)=0 \, \hbox{ with } \ G_{n}(y):=D(y^*-y)-n\,\tildeb(y) \,. 
\end{equation}
Then,

\smallskip

\noindent $(i)$ If $\eta=0$, then for any $n\in \mathbb{N}$, 
Equation \eqref{root} has a unique simple root $y_{n}$, which belongs to $[0,y^*]$. 
In addition, the sequence $(y_{n}: n\in \NN^*)$ of the roots decreases to $0$.

\smallskip

\noindent $(ii)$ If $\eta>0$ then Equation \eqref{root} has no root 
for $n> {Dy^*\over \eta}$ and admits a simple root $y_{n}$ for $n\le {Dy^*\over \eta}$. 
\end{lemma}


\begin{proof} Let us fix $n$. Obviously for $n=0$, Equation \eqref{root}  has  the trivial root $y_{0}=y^*$,
so we restrict $n$ to be in $\NN^*$. By assumption, the function 
$G_{n}$ is strictly decreasing 
so for each $n\in \NN^*$ there exists at most one root. 
Note that for all such $n$ we have
$G_{n}(y^*)<0$ and so there is no root to \eqref{root} in the set 
$[y^*,\infty)$.

\medskip

Assume $\eta=0$. For all $n$ we have
$G_{n}(0)>0$ then there exists a unique root which is denoted 
$y_n$, so it satisfies $G_{n}(y_{n})=0$. On the other hand, we have 
$$
G_{n+1}(y_{n})= 
D(y^*-y_{n})-(n+1)\,\tildeb(y_{n}) = -\,\tildeb(y_{n})\,.
$$ 
Then, $G_{n+1}(0)>0, G_{n+1}(y_{n})<0$. We deduce $0<y_{n+1}<y_{n} <y^*$.
Let $y_\infty=\lim\limits_{n\to \infty} y_n$. By continuity
$$
D(y^*-y_\infty)-n\,\tildeb(y_\infty)=
\lim\limits_{n\to \infty}D(y^*-y_{n})-n\,\tildeb(y_{n})=0.
$$
Then, necessarily $\tildeb(y_\infty)=0$ and so $y_\infty=0$, and $b(y_{n})\sim {Dy^* R\over n}$ as $n\to \infty$.

\medskip

Assume $\eta>0$. Then for all $n>Dy^*/\eta$
we  have $G_{n}(0)< 0$ and so \eqref{root} has no solution. 
Hence the same
argument as before gives the existence of a finite set of roots
$(y_n: 1\le n\le \lfloor Dy^*/\eta \rfloor)$  decreasing with $n$, where $\lfloor Dy^*/\eta \rfloor$  the biggest integer 
that is smaller or equal to $Dy^*/\eta$.

\medskip

Note that when $n_0=Dy^*/\eta$ then $y_{n_0}=0$ and if $Dy^*<\eta$,
there is no root.  
\end{proof}

\bigskip
From Lemma \ref{roots1}, 
 we know that the set 
 \begin{equation}
 \label{moche}
\moche = \{y\in \RR_+: \, \exists n\in \NN^*, D(y^*-y) -{\tilde b(y) n} =0\}.
\end{equation}
 is a countable set
included in $[0,y^*)$. If $\eta=0$, it is infinite and accumulates at $0$
and if $\eta>0$,  it is finite.

\bigskip

In the  sequel, when we refer to $y_n$, we will assume implicitly that it exists,
namely we are in the case  $\eta=0$  or $\eta>0$ but 
$n\in \{1,..\lfloor Dy^*/\eta \rfloor\}$.

\medskip

\begin{corollary}
\label{masinva}
$(i)$ The set $\NN\times [0,y_1]$ is invariant for the process $Z^{T_0}=(Z(t): t\le T_0)$
up to extinction, that is if $Z(0)=(N(0), Y(0))\in \NN\times [0,y_1]$ then
$Z(t)=(N(t), Y(t))\in \NN\times [0,y_1]$ for all $t\le T_0$. 

\medskip



\noindent $(ii)$ The set $\NN\times [0,y_n]$ is invariant for the process 
$Z^{T_{<n}}=(Z(t): t\le T_{<n})$.
\end{corollary}

\begin{proof}
Let us show the first part. We have 
$Z(t)=(N(t), Y(t))\in \NN^*\times [0,y_1]$ for all $t< T_0$ 
because $dY(t)/dt\le 0$ when $y\ge y_1$ and $n\ge 1$, and so if
the trajectory arrives to $y_1$ the variable $Y(t)$ immediately decreases.
For $n=0$ it is evident, because we stop the process at $T_{0}$, and so
$Y(T_0)=Y(T_0^-)$ but $Y(T_0^-)\le y_1$ since $N(T_0^-)=1$.
 
\medskip

 Part $(ii)$ is shown in a similar way
as $(i)$.
\end{proof}

\bigskip 
Let us state  a useful lemma. 


\begin{lemma}\label{bornemort}
For any $n_{0} \in \mathbb{N}^*$, there exists $t_{0}>0$ such that
$$
\inf_{y\in[0,y^{*}],\, 1\le n\le
n_{0}}\proba_{(n,y)}\big(T_{0}<t_{0}\big)>0\;.
$$
\end{lemma}
\begin{proof}
This follows at once from the fact that the population process is stochastically dominated by a birth and death process with birth rate $b(y^*)$ and death rate $D+ d(y^*)$.
\end{proof}

\begin{theorem}
We have extinction of the population almost surely, namely for any $y\in
[0,y^{*}]$ and any integer $n$,
$$
\proba_{(n,y)}\big(T_{0}<\infty\big)=1\;.
$$
\end{theorem}

\begin{proof} Remark that an obvious comparison theorem as used in the previous proof cannot be applied. Indeed the birth and death rates $b(y^*)$ and  $D+ d(y^*)$ could correspond to a supercritical case. The effect of the chemostat through the nutrient is a regulation of the population.

We will exhibit an integer $n_{0}$ such that the population process will spend an infinite amount of time below $n_{0}+1$. 

\medskip
Let $\ty>0$ be such that $b(\ty)<D+d(\ty)$. Note that by
monotonicity, for all $y\in[0,\ty]$ we have   $b(y)<D+d(y)$. 

Let us define 
$n_{0}$  as an integer such that $y_{n_{0}}<\ty$ if some exists or equal to
$[Dy^*/\tilde b(0)]+1$ otherwise. 

Let $\tau$ be the random time defined by
$$
\sup_{t\ge\tau}N(t)\le n_{0}\;.
$$ 
Then, it follows from Lemma \ref{bornemort} and the Markov property that if $\proba_{(n,y)}\big(\tau<\infty\big)>0$, then
$$
\proba_{(n,y)}\big(T_{0}<\infty\,\big|\, \tau<\infty\big)=1\;.
$$

Assume now $\proba_{(n,y)}\big(\tau=\infty\big)=1$. Let  $\tau'$ be the
random time defined by 
$$
\inf_{t\ge\tau'}N(t)> n_{0}\;.
$$
Assume $\proba_{(n,y)}\big(\tau'<\infty\big)>0$. Let $y_{0}(t)$ be the
solution of the differential equation
$$
\frac{dy_{0}}{dt}=D(y^{*}-y_{0}(t))-(n_{0}+1)\tilde b(y_{0}(t))\;,
$$
with initial condition $y_{0}(0)=y^{*}$. Let $t_{1}>0$ be the finite
solution of
$$
y_{0}(t_{1})=\ty\;.
$$

The time $t_{1}$ is finite from the choice of $n_{0}$.
It is easy to verify that for any integer valued measurable function
$n(t)\ge n_{0}+1$, the solution $y(t)$ of the equation 
$$
\frac{dy}{dt}=D(y^{*}-y(t))-n(t)\tilde b(y(t))\;,
$$ 
with initial condition $y(0)\le y^{*}$, satisfies
$$
y(t)\le \ty
$$
for any $t\ge t_{1}$. On the set  $\big\{ \;\tau'<\infty\big\}$,
the process $(N(t), t\geq  
\tau'+t_{1})$ is dominated by a linear birth and
death process with birth rate $b(\ty)$  and death rate $D+d_{0}(\ty)$
(from the monotonicity of the functions). This birth and death chain attains $n_{0}$
almost surely in finite time since $b(\ty)<D+d_{0}(\ty)$, see
\cite{vandoorn}. Hence  on $\big\{ \;\tau'<\infty\big\}$, the process $(N(t), t\geq  
\tau'+t_{1})$  should also attains
$n_{0}$ in finite time. This contradicts our assumption
$\proba_{(n,y)}\big(\tau'<\infty\big)>0$. 

It remains to consider the case $\tau=\tau'=\infty$ almost surely. In this case there
exist two infinite sequences of random times
$$
T_{1}<S_{1}<T_{2}<S_{2}<\cdots
$$ 
such that
$$
N(t)\le n_{0}\quad \mathrm{for}\; t\in[T_{j},S_{j})\;,\quad 
N(t)>n_{0}\quad \mathrm{for}\; t\in[S_{j},T_{j+1})\;.
$$
Since we visit the set $\{N\le n_{0}\}$ infinitely many times, and  at
each visit we have a uniformly positive probability of extinction, it
follows by the Markov property and the Borel Cantelli Lemma that 
$$
\proba_{(n,y)}\big(T_{0}<\infty\,\big|\,\tau=\tau'=\infty\big)=1\;.
$$
\end{proof}

One of our main objectives of this work is to study the processes
up to the moment the population is extinct $Z^{T_0}=(Z(t): t\le T_0)$,
or before the moment of extinction $Z^{T_0^-}=(Z(t): t< T_0)$.
All the statements related to quasi-stationary distributions depend
on $Z^{T_0^-}$.

\section{Existence of Quasi-Stationary Distributions}

A quasi-stationary distribution $\nu$ (with respect to the absorbing time $T_0$) is a probability measure
defined on  $\NN^*\times \RR_+$ that verifies
\begin{equation}
\label{QSDI}
\PP_\nu(Z(t)\in D \, | \, T_0>t)=\nu(D) \,, \; \;
\forall D\in \B(\NN^*\times \RR_+)\,. 
\end{equation}
It is known that  starting from a quasi-stationary distribution, the time of absorption is exponential,
that is $\PP_\nu(T>t)=e^{-\lambda t}$ where $\lambda=\lambda(\nu)>0$. 
Let ${\cal M}_b=\{f:\NN^*\times \RR_+\to \RR \hbox{ bounded and measurable}\}$. 
Equation \eqref{QSDI} can be written as,
\begin{equation}
\label{QSDI1}
\exists\, \lambda>0: \;\, 
\EE_\nu(f(Z(t)),\, T_0>t)=e^{-\lambda t}\int_{0}^\infty f(y)\nu(dy)  \; \;
\forall t>0, f\in {\cal M}_b.
\end{equation}

Denote by 
$$
\kappa_n=\nu(\{n\}\times \RR_+)=\PP_\nu(N(0)=n)\,,
$$
and the probability measure conditioned to have $n$ individuals by
$$
\nu_n(B)=\nu(\{n\}\times B \,| \, \{n\}\times \RR_+)=\PP_\nu(Y(0)\in B \, | \, N(0)=n) \;\;
\forall B\in \B(\RR_+)\,.
$$
Then $\nu(D)=\sum\limits_{n\in \NN^*} \kappa_n \nu_n(D\cap \{n\}\times \RR_+)$,
$\nu_n(\RR_+)=1$ $\; \forall n\in \NN^*$ and $\sum\limits_{n\in \NN^*} \kappa_n=1$.

\medskip

In order that the probability measure $\nu$ is a quasi-stationary distribution, it must satisfy 
the infinitesimal condition deduced from \eqref{QSDI1}, which is given by,
$$
\exists\, \lambda>0: \;\,
\sum\limits_{n\in \NN^*} \kappa_n
\int_0^\infty d\nu_{n}(y) \bigg[{\cal L}f(n,y) - \lambda f(n,y)\bigg] = 0
\;\; \forall  f\in {\cal M}_b.
$$
In terms of the adjoint operator $\mathcal{L}^{\dagger}$ of 
$\mathcal{L}$ defined in \eqref{generator} the equation to be a quasi-stationary distribution is,
\begin{equation}
\label{spectral}
\exists\, \lambda>0: \;\; \mathcal{L}^{\dagger}\nu=-\lambda\,\nu\,.
\end{equation}

\medskip

\begin{theorem} Assume there exists  $0\le \sigma<1$
such that $\limsup\limits_{y\searrow0}y^{\sigma}d(y)<\infty$.
Then  there exists a quasi-stationary distribution.
Moreover there exists at least a quasi-stationary distribution such that $Y$ is supported
in $[0,y^*]$.
\end{theorem}

\begin{proof} It suffices to show the existence of a quasi-stationary distribution. In fact,
the last part of the statement follows when the existence proof is applied
to the process $Z$ taking values in the invariant set 
$\mathbb{N}\times [0,y^*]$.

\medskip
 
The idea for showing the existence is to use the abstract 
Theorem 4.2 proved in \cite{quatres}. 
We assume $y\in[0,y^{*}]$. We define a function 
$$ 
\varphi_{1}(y,n)=\mathbf{1}_{n\ge 1}\;. 
$$ 
Then a simple computation (it is the same computation as 
for a birth and death process since $\varphi_{1}$ does not depend on $y$) 
leads to 
$$ 
{\cal L}\varphi_{1}(y,n)= 
\left
\{\begin{array}{cc} 
0&\mathrm{if}\;n>1\;,\\ -D-d(y)&\mathrm{if}\;n=1\;. 
\end{array}
\right. 
$$ 
Therefore if $d_{0}$ is bounded above by a constant $d_{0}^*$, then 
$$ 
e^{-t(D+d_{0}^*) } \varphi_{1}\le e^{t{\cal L}} \varphi_{1}\le \varphi_{1}\;. 
$$ 
If $d_{0}$ is not bounded, we will prove the  lower bound
$$
\inf_{n_{0}\ge 1,\,y_{0}\in]0,y^{*}]}\proba_{n_{0},y_{0}}(N(1)\ge
    1)\ge Q\;.
$$ for a constant $Q>0$.   
By the Markov property it suffices to prove that  there exists $Q>0$
such that 
$$ 
\inf_{y_{0}\in[0,y^{*}]}\proba_{1,y_{0}}(N(t)=1,\;\forall\;0\le t \le
1)\ge Q\;.
$$
Since  $N=1$ on the whole time interval $[0,1]$, $y(t)$ satisfies the
differential equation 
\begin{equation}\label{eqy}
\frac{dy}{dt}=D(y^{*}-y)-\tilde b(y)\;,
\end{equation}
with initial condition $y_{0}$.
Since there neither birth, nor death on the time interval $[0,1]$
we get 
$$
\proba_{1,y_{0}}(N(t)=1,\;\forall\;0\le t \le 1)=
e^{-\int_{0}^{1}\big(d(y(t))+b(y(t)\big)\,dt}\;.
$$
Since  $b(y)$ is bounded uniformly in $y$, the above quantity 
does not vanish if 
$$
\int_{0}^{1}d(y(t))\,dt<\infty\;.
$$
It is easy to show that there is a constant $c>0$ such that for any
$y^{*}\ge y_{0}\ge 0$, we have $y(t)\ge ct$ for any $t\in[0,1]$.
Since $d(y)\le \mathcal{O}(1) y^{-\sigma}$ for $y$ small, we get
$$
\int_{0}^{1}d(y(t))\,dt\le \mathcal{O}(1)\int_{0}^{1}y(t)^{-\sigma}\,dt<\infty\;.
$$ 
It now follows immediately that
$$
Q\;\varphi_{1}\le e^{{\cal L}} \varphi_{1}\le \varphi_{1}\;.
$$
 The second function is 
$$ 
\varphi_{2}(y,n)=\mathbf{1}_{n\ge 1}\;e^{a(y)n}\;, 
$$ 
with $a(y) = \alpha y + a_{0}$, 
$\alpha>0$ and $a_{0}>0$. A simple computation for $n>1$ (it is the same 
computation as for a birth and death process since $\varphi_{2}$ 
does not depend on $y$) leads to,  
$$
{\cal L}\varphi_{2}(y,n)= n\,e^{a(y)n}\, \Xi(n,y)\;,
$$
where
$$
\Xi(n,y)= 
 b(y)\left(e^{a(y)}\!-\!1 \!- {\alpha n\over R}\right)\!+
(D\!+\!d(y))\left(e^{-a(y)}\!-\!1\right) \!+\! D \alpha (y^*\!-\!y) 
\!-\! \alpha \eta {\bf 1}_{y>0}\;.
$$ 

Let us show that there exists $A>0$ and $N_{0}$ such that for all $n>N_{0}$, 
$$
\sup_{y\in [0,y^*]} \Xi(n,y) \leq - A.
$$
Define $\zeta(y)= (D+d(y))\left(e^{-a(y)}\!-\!1\right) \!+\! D \alpha (y^*-y) 
-\alpha \eta {\bf 1}_{y>0}$. 
We choose $a_{0}$ and $\alpha$ such that $D \left(e^{-a_{0}} -1\right) +\! D \alpha y^*<0$. Then  $\zeta(0)<0$. It follows  that $\forall y\in[0, y^*]$, $\zeta(y)<- A$ for some $A>0$.

Consider $\hat N_{0}$ such that $e^{a(y^*)}-1 - {\alpha \hat N_{0}\over R} <0$. We still have for $n>\hat N_{0}$ that $e^{a(y^*)}-1 - {\alpha n\over R} <0$.

Then  for  $n>\hat N_{0}$, we get 
$ \Xi(n,y)\leq -A $.
Therefore for any $C>0$ there 
exists $N(C)$ such that for any $n>N(C)$, 
$$ 
{\cal L}\varphi_{2}(y,n)\le -C\,\varphi_{2}(y,n)\;. 
$$

Therefore, for any $C>0$ there exists $\Gamma(C)>0$ (finite) such that 
$$ 
{\cal L}\varphi_{2}\le -C\,\;\varphi_{2} +\Gamma(C)\;\varphi_{1}\;. 
$$ 
(the 
estimate for $n=1$ is by direct computation taking $\Gamma(C)$
adequately large enough).  
Hence,
$$ 
e^{t{\cal L}} \varphi_{2}\le 
e^{-tC}\;\varphi_{2}+\frac{\Gamma(C)}{C}\varphi_{1}\;. 
$$ 
In the case where $d$ is bounded above by $d^{*}<\infty$, we now
choose $C>D+d^{*}$  
and $t=1$ and apply Theorem 4.2 in \cite{quatres}. 

In the case $d$ unbounded, we choose $C>-\log Q$ and apply Theorem 4.2 in
\cite{quatres}.  

\end{proof}


\section{Properties of the Quasi-Stationary Distributions}

\begin{proposition} 
\label{yaglom} 
Any quasi-limiting distribution has support in $\mathbb{N}^*\times[0,y^*]$. 
\end{proposition}

\begin{proof} 
We recall that a quasi-limiting distribution $\nu$ is a probability
measure on $\NN^*\times \RR_+$ such that for some initial point  
$(n_{0}, y_{0})\in \NN^*\times \mathbb{R}_{+}$ 
$$
 \nu(A)= \lim\limits_{t\to \infty}
\PP_{(n_0, y_0)}(Z(t)\in A \, | \, T_0>t) \;,\; \forall A\in \B(\NN^*\times \RR_{+}). 
$$
If $y_{0}\in [0,y^*]$,  the assertion follows from Proposition \ref{defproc}. Let us now assume that $y_{0}>y^*$. 
We introduce the function $t\to v(t)$ defined by 
$$
{dv\over dt}(t) = D(y^* -v) - \tilde b(y^*)\ ; v(0)=y_{0}.
$$ 
Let $\tau_{0}(y_{0})$ be 
defined by $v(\tau_{0}(y_{0})) = y^*$. For $t\leq T_{0}\wedge \tau_{0}(y_{0})$, we have $Y(t)\leq v(t)$.
It follows from Proposition \ref{defproc}  that 
$\mathbb{P}_{(n_0,y_0)}(Y(t)>y^*| T_{0}>t) = 0 \; ,\,\forall t\geq \tau_{0}(y_{0})$, 
which concludes the proof. 
\end{proof}

\bigskip 

\medskip

\begin{theorem} 
\label{abscontqsd}
For all $n\in \NN^*$  and any quasi-stationary distribution,  the probability measure $\nu_n$ 
is absolutely continuous with respect to the Lebesgue measure, 
with  $C^{\infty}$-density   on  the set 
$\,\RR_+\setminus \{0,y_n\}$. 
\end{theorem}

The proof of this theorem is obtained from the following lemmas. Recall that the set $\moche$ has been defined in \eqref{moche}.

\begin{lemma} 
\label{AC} 
For all $n\in \NN^*$, the measure 
$\nu_{n}$ satisfies 
\begin{equation}
\label{lebdec}
d\nu_{n}=c^{n}_{0}\;
\delta_{0}+  \sum_{j\in \NN^*} \bigg(c^{n}_{j}\;
\delta_{y_{j}}+u_{n}(y)\, dy\bigg)
\end{equation}
with $c^{n}_{j}\ge 0$ for $j\in \NN$ and $u_{n}$ is the density of the
absolutely continuous part of $\nu_n$ (so it is a
non-negative integrable function) and it is a $C^{\infty}$ 
function outside $\moche\cup \{0\}$.
\end{lemma} 

\begin{proof} The measures $\nu_{n}$ satisfy in the sense of 
distributions the differential equations 
\begin{eqnarray}
\nonumber 
&{}& \partial_{y}\left(\left(D(y^*-y)-n\,\tilde b(y)\right)\;\nu_{n}\right)\\
\label{eqnu}
&{}& \!=\!b(y)n\,\nu_{n-1}\!+ \!(D\!+\!d(y))n\,\nu_{n+1} 
\!-\!(b(y)\!+\!D\!+\!d(y))n\,\nu_{n}\!+\!\lambda \nu_{n}. 
\end{eqnarray} 
Since the right hand side is a measure, we 
conclude by a recursive argument that the measures 
$\nu_{n}$ have a $C^{\infty}$ density on  $(\moche\cup \{0\})^c$. 
This also shows that these measures have no singular part, and the 
Lebesgue decomposition theorem gives relation \eqref{lebdec}. 
\end{proof}

\medskip

\begin{lemma} 
\label{lemme5.4}
Let $I$ be an open interval included in $(\moche\cup \{0\})^c$. If there 
exists $n\in \NN^*$ such that $\nu_{n}(I)=0$, then $\nu_{j}(I)=0$ 
for all $j\in \NN^*$. 
\end{lemma} 

\begin{proof}
Assume that there is an interval $I$ not intersecting $\moche\cup \{0\}$ 
such that for some integer $n\in \NN^*$,
$\nu_{n}(I)=0$ vanishes on $I$. From \eqref{eqnu} we deduce that
$\nu_{n+1}$ and $\nu_{n-1}$ also vanish at $I$. Therefore, we conclude recursively 
that $\nu_{j}(I)=0$ for all $j\in \NN^*$. 
\end{proof}

\medskip

\begin{lemma} 
The probability measure $\nu_{n}$ is absolutely continuous on $\RR_+\setminus \{0,y_n\}$ 
and its density is bounded on any compact set contained in $\RR_+\setminus \{0,y_n\}$. 
\end{lemma} 

\begin{proof} 
Let us show that if $j\neq n$, $\nu_{n}$ cannot have a Dirac 
mass in $y_{j}$. We will do it by contradiction, so assume it does. 
Let $f$ be a $C^{\infty}$ function with compact support 
containing $y_{j}$ and such that its support does not contain any other point
of $S\cup \{0\}$ except $y_{j}$. By using formula \eqref{lebdec} we get, 
\begin{eqnarray} 
\nonumber &{}& 
-f'(y_{j})\left(\!D(y^*\!-\!y_{j})- n\,\tilde b(y)\right) -
\int \!\!f'(y) \left(D(y^*\!-\!y)-n\,\tilde b(y)\right)\psi_{n}(y)dy\\ 
\nonumber 
&{}& =\int 
f(y)\;\big( b(y)\,n\,d\nu_{n-1}(y)+ (D+d(y))\,n\,d\nu_{n+1}(y)\big)\\ 
\label{eqfamily} 
&{}&\;\; +\int f(y)\,\big(\lambda -(b(y)+D+d(y))\,n\big)\,d\nu_{n}(y)\;. 
\end{eqnarray} 

It is not difficult 
to construct a sequence of functions $(f_{q}: q\in \NN^*)$ contained in $C^{\infty}$ 
with support in a fixed 
small enough neighborhood of $y_{j}$ and such that, 
$$ 
f'_{q}(y_{j})\!=\!1; \;
\lim_{q\to\infty}\sup_{y\in \RR_+}|f_{q}(y)|\!=\!0 ; \; \sup_{y\in \RR_+} 
|f'_{q}(y)|\!\leq \!1 \ \forall q ; \;
\lim_{q\to\infty} |f'_{q}(y)|\! =\! 0\; \forall y\neq y_{j}. 
$$ 
This leads to a contradiction, when 
we take $f=f_{q}$ in \eqref{eqfamily} and make $q$ tend to infinity.

\medskip

Now, using again equation \eqref{eqnu}, we deduce easily the boundedness of the 
density of $\nu_{n}$ outside a neighborhood of $\{y_{n},0\}$. 
\end{proof}

\begin{lemma} 
\label{56}
The probability measure $\nu_{n}$ cannot have a Dirac mass in $y_{n}$. 
\end{lemma}

\begin{proof}
 Assume it does. Then in a neighborhood of $y_{n}$, by using \eqref{lebdec}
since $c_j^n=0$ for $j=0,n$ and by writing $c_{n}=c_n^n$, we can write
$$
d\nu_{n}=c_0^n \delta_{0}+c_{n}\delta_{y_{n}}+u_{n}(y)\, dy\,. 
$$ 

Let $f$ be a function $C^{\infty}$ with compact support containing 
$y_{n}$ but that does not contain $0$. We have 
\begin{eqnarray*} 
&{}& 
-\int f'(y) \left(D(y^*-y)-n \,\tilde b(y)\right)\;u_{n}(y)\,dy \\  
&{}&
=\int f(y)\;\big( b(y)\,n\,d\nu_{n-1}(y)+ (D+d(y))\,n\,d\nu_{n+1}(y)\big)\\ 
&{}& \;\; 
+\int f(y)\;\big(\lambda -(b(y)+D+d(y))\,n\big)\,u_{n}(y)\,dy \\
&{}& \;\;
+c_{n}\,f(y_{n})\,\big(\lambda -(b(y_{n})+D+d(y_{n}))\,n\big)\;. 
\end{eqnarray*} 
We now 
construct a sequence $(f_{q}: q\in \NN^*)$ 
of $C^{\infty}$ functions with support in a fixed small 
enough neighborhood of $y_{n}$ such that for some constant $C'$,
$$
f_{q}(y_{n})=1\ \forall q\ ; \; \sup_{y\in \RR_+}|f_{q}(y)| = 1\ ; \; \lim_{q\to 
\infty}|f_{q}(y)| =0\; \forall y\neq y_{n} 
$$ 
and 
$$ 
\ \sup_{y\in \RR_+}|y-y_{n}| 
|f'_{q}(y)|\leq C'\; \forall q\ ;\; \lim_{q\to \infty} (y-y_{n})f'_{q}(y)=0
\ \forall y\neq y_{n}. 
$$ 
Such a sequence can be easily constructed.

\medskip 

Recall that $\nu_{n-1}$ and $\nu_{n+1}$ are absolutely continuous with 
$C^{\infty}$ density near $y_{n}$. We conclude that if $c_{n}\neq 0$ 
$$ 
\lambda=(b(y_{n})+D+d(y_{n}))\,n\;. 
$$

In the case $d$ is constant, it is known that the eigenvalue $\lambda$ satisfies 
$\lambda=\kappa_n(D+d)<D+d$, and so it is strictly less than 
$(b(y_{n})+D+d(y_{n}))\,n\;$ and we obtain a contradiction.

\bigskip 

If $d$ is not constant, the proof of the contradiction is more intricate.

Coming back to Equation \eqref{eqnu} for $u_{n}$ in a neighborhood of $y_{n}$ but outside 
that point we get 
$$ 
\partial_{y}\big(G_{n}(y)\, u_{n}(y)\big)= 
f_{n}(y)\,G_{n}(y)\, u_{n}(y)+h_{n}(y) $$ with (see Lemma \ref{roots1}) $$ G_{n}(y)=D(y^*-y)-n\,
\tilde b(y) =\beta_{n}(y-y_{n})+\mathcal{O}\big((y-y_{n})^{2}\big),
$$ 
with $\beta_{n}<0$, and 
$$ 
f_{n}(y)=n\,\frac{b(y_{n})-b(y)+d(y_{n}) 
-d(y)}{G_{n}(y)}=\mathcal{O}(1)\;, 
$$ 
and 
$$ 
h_{n}(y)=b(y)\,n\,u_{n-1}(y)+ 
(D+d(y))\,n\,u_{n+1}(y). 
$$

and recall that $h_{n}(y)$ is $C^{\infty}$ near $y_{n}$. The only solution which is 
integrable near $y_{n}$ is given by 
$$ 
u_{n}(y)=\frac{1}{G_{n}(y)} \exp\left(\int^{y}_{y_{n}}f_{n}(s)ds\right)\;\int_{y_{n}}^{y}e^{-\int^{s}_{y_{n}}f_{n}(w)dw} 
\;h_{n}(s)\,ds \;. 
$$


If $h_{n}>0$ on a subset of positive measure of a small neighborhood of $y_{n}$, we 
have $u_{n}<0$ which is a contradiction.

Therefore, $h_{n}$ must vanish on both sides of $y_{n}$. By the above result on the 
support, we conclude that $u_{n}$ vanishes in a neighborhood of $y_{n}$ as well as 
all the $\nu_{j}$ with $j\neq n$ (see Lemma \ref{lemme5.4}). In particular, if we consider the equation for 
$\nu_{n+1}$ in this neighborhood (see \eqref{eqnu}), we get 
$$ 
0=b(y_{n})\, c_{n}\, \delta_{y_{n}} 
$$ 
which contradicts $c_{n}\neq 0$. 
\end{proof}

\begin{theorem}
\label{positiv} On $(0,y_{1})$, the density of $\nu_{n}$ satisfies 
$u_{n}>0$ except perhaps in $y_{n}$. \end{theorem}

The proof uses two lemmas.

\begin{lemma} 
If $\nu_{n+1}$ or $\nu_{n-1}$ has a support dense in $(0,y_{1})$ then 
$u_{n}$ can be $0$ only in $y_{n}$. 
\end{lemma} 

\begin{proof} 

The function $G_{n}$ has a simple zero in $y_{n}$. Assume that for $z\in (0,y_{1})$, 
$z\neq y_{1}$, $u_{n}(z)=0$ (and $G_{n}(z)\neq 0$).

Computation as in the proof of Lemma \ref{56} gives 
$$ 
u_{n}(y)=\frac{1}{G_{n}(y)} \exp\left(\int^{y}_{z}f_{n}(s)ds\right)\;\int_{z}^{y}e^{-\int^{s}_{z}f_{n}(w)dw} 
\;h_{n}(s)\,ds \;. 
$$

The conclusion follows. 
\end{proof}

\medskip

Let us show that the process is irreducible up to extinction and before $y_1$.

\medskip

\begin{lemma} 
\label{irreducp}
Starting from any initial condition on $\NN^*\times (0,y_1)$ 
the process $Z$ has dense support on $\NN\times (0,y_1)$, that is,
\begin{eqnarray*}
&& \forall m\in \NN, \, \forall y'\in (0,y_1),\, \forall \gamma>0,\,
\exists t(y')>0 \hbox{ such that } \forall t> t(y'),\\
&& 
\forall (y_0,n_0)\!\in \!\NN^*\times (0,y_1):\;
\mathbb{P}_{(n_0,y_0)}(N(t)\!=\!m,Y(t)\!\in \! (y'\!-\!\gamma, y'\!+\!\gamma))\!>\!0.
\end{eqnarray*}
\end{lemma}

\begin{proof} 
In the proof we will assume $y_0<y'$ (the case $y_0\in (y',y_1)$ is
shown similarly). 
Let $\beta>0$ be smaller than $\min(y_0, y'-y_0)/2$ 
and let $\tilde t>0$ be fixed. 
Then, there is $\epsilon=\epsilon(n_0, \beta)$ 
such that the following event has a strictly positive probability:

\smallskip

\noindent $-\;$ On the interval time $[0,\epsilon]$ there are exactly 
$n_0-1$ deaths and there is no other jump of $N$, and so $N(t)$ decreases from $n_0$ 
to $N(\epsilon)=1$; 

\smallskip

\noindent $-\;$ $Y(\epsilon)$ belongs to $(Y(0)-\beta, Y(0)+\beta)$;

\smallskip

\noindent $-\;$ On the interval of time $[\epsilon, \epsilon+{\tilde t}]$ there
is no jump of $N$ (no birth and no death);

\smallskip

\noindent $-\;$ On the interval of time 
$[\epsilon+{\tilde t}, \epsilon+{\tilde t}+\beta]$ there are exactly $m-1$ births
and no other jump when $m>1$, there is no jump if $m=1$, or there is a unique
death and no other jump when $m=0$;

\smallskip

\noindent $-\;$ $|Y(\epsilon+{\tilde t})+\beta)-Y(\epsilon+{\tilde t})|<\gamma/2$.

\medskip

For $t\in [\epsilon, \epsilon+\tilde t)$ we have $N(t)=1$, then 
in this interval of time and before the process $Y(t)$ has attained 
$y'$, the derivate
$$
\frac{dY(t)}{dt}=D(y^*-Y(t))-\tildeb(Y(t)) 
$$
is bounded below by $D(y^*-y')-\tildeb(y')$.
Take ${\tilde t}= (y'-Y(0)+\beta)/(D(y^*-y')-\tildeb(y'))$, let us see that
the number $t(y')=\epsilon+{\tilde t}'$ makes the job. In fact, we 
have ensured that in a time
smaller or equal to $t(y')$ we have attained $\{m\}\times [y'-\gamma, y'+\gamma]$.
For any time bigger than $t(y')$ it suffices to modify slightly the above argument
and allow a sequence of jumps up to the moment that $Y(t)$ has negative derivate 
and in this way we can postpone the time of attaining the set 
$\{m\}\times [y'-\gamma, y'+\gamma]$ from $t(y')$ to a prescribed time $t>t(y')$. 
\end{proof}

\medskip

\begin{proposition}
\label{irredqsd}
For all $n\in \NN^*$, the probability measure $\nu_{n}$ has a support dense in
$(0,y_{1})$. 
\end{proposition}

\begin{proof}
Denote by $n\otimes\nu_n$ the probability measure defined on $\NN\times \RR_+$ by
$n\otimes\nu_n(\{m\}\times B)=\delta_n(m)\nu_n(B)$ for all $B\in \B(\RR_+)$.
From Lemma \ref{irreducp} and the quasi-stationary distribution  property \eqref{QSDI1} we have
\begin{eqnarray*}
&&\PP_{n\otimes\nu_n}(N(0)\!=\!m, Y(0)\!\in \!(y'\!-\!\gamma, y'\!+\!\gamma)=
e^{-\lambda t}\PP_{n\otimes\nu_n}(N(t)\!=\!m, Y(t)\!\in \!(y'\!-\!\gamma, y'\!+\!\gamma)\\
&& \; =
e^{-\lambda t}\int \! \PP_{(n,y_0)}(N(t)\!=\!m, Y(t)\!\in \! (y'\!-\!\gamma, y'\!+\!\gamma))
d\nu_n(y_0)>0.
\end{eqnarray*}
Then, the result follows.
\end{proof}

\medskip

\section{Bound on the asymptotic survival rate} 





\begin{theorem}
Let $\lambda$ be the exponential  extinction rate associated  with a quasi-stationary distribution $\nu$. Then
\label{bornelambda}
\begin{equation}
\label{rate}
\lambda< \inf_{n}(n(b(y_{n})+D+d(y_{n})).
\end{equation}
\end{theorem}

\begin{proof}
In Section 3 in \cite{quatres}, it was shown that this extinction rate is an eigenvalue of the dual problem associated with the probability measure
$${\cal L}^\dag \nu  = -\lambda \nu.
$$
For simplicity, we will prove inequality \eqref{rate} in case of $n=1$.
The general case proved in a similar way
is left to the reader. 

We introduce the notation
$$
G(y)=D(y^{*}-y)-\tildeb(y)\,
$$
$$
H(y)=\big(D+d(y)\big)\,
$$

The quasi-stationary distribution equation for $n=1$ is given by
$$
-\frac{d}{dy}\big(G(y)\,u_{1}(y)\big)-
\big(b(y)+H(y)\big)u_{1}(y)+ 2 H(y)\,u_{2}(y)+\lambda
\, u_{1}(y)=0\;.
$$
Let $a\in]0,y_{1}[$. We have for $y\in]0,y_{1}[$
$$
\frac{d}{dy}\left(G(y)\,u_{1}(y)\,e^{\int_{a}^{y}(b(s)+H(s)-\lambda)/G_{1}(s)
ds}\right)=e^{\int_{a}^{y}(b(s)+H(s)-\lambda)/G(s)
ds}\,2\, H(y)\,u_{2}(y)\;,
$$
and integrating between $a$ and $y$ we get
$$
\big(G(y)\,u_{1}(y)\,e^{\int_{a}^{y}(b(s)+H(s)-\lambda)/G(s)
ds}-G(a)u_{1}(a)
$$
$$
= 2\int_{a}^{y}e^{\int_{a}^{\sigma}(b(s)+H(s)-\lambda)/G(s)
ds}\,H(\sigma)\,u_{2}(w)\,dw \ge 0\;.
$$
Therefore
$$
u_{1}(y)\ge \frac{G(a)u_{1}(a)}{G(y)}\;
e^{-\int_{a}^{y}(b(s)+H(s)-\lambda)/G(s)ds}\;.
$$
Let us study more carefully the quantity  $\;e^{-\int_{a}^{y}(b(s)+H(s)-\lambda)/G(s)ds}$. We recall that $G$ only vanishes at $y_{1}$ and that $G$ is decreasing since $b$ is increasing.  
By a simple computation one gets
$$b(s)+H(s)-\lambda)/G(s) = \frac{b(y_{1})+H(y_{1})-\lambda}{(s-y_{1})\,G'(y_{1})}\,+\mathcal{O}(1)\;,
$$
and thus 
$$
\int_{a}^{y}(b(s)+H(s)-\lambda)/G(s)
ds=\frac{b(y_{1})+H(y_{1})-\lambda}{G'(y_{1})}\,
\log|y-y_{1}|+\mathcal{O}(1)\;.
$$
Finally, since  $G(y_{1}) = 0$, and $G'(y_{1})<0$, we get for $y\le y_{1}$
$$
G(y)=(y-y_{1})\,G'(y_{1})+\mathcal{O}\big((y-y_{1})^{2}\big)=
|y-y_{1}|\big(|G'(y_{1})|+\mathcal{O}(|y-y_{1}|)\big)\;.
$$
Therefore
\begin{eqnarray*}
u_{1}(y)&\ge &\frac{G(a)u_{1}(a)}{G(y)}\,
|y-y_{1}|^{-(b(y_{1})+H(y_{1})-\lambda)/ G'(y_{1})}\,
e^{\mathcal{O}(1)}\\&\geq& \frac{G(a)u_{1}(a)}{|y-y_{1}|\,|G'(y_{1}|}\,
|y-y_{1}|^{-(b(y_{1})+H(y_{1})-\lambda)/ G'(y_{1})}\,
e^{\mathcal{O}(1)}\; .
\end{eqnarray*}
 As we have  $u_{1}(a)>0$ by Theorem \ref{positiv},
the integrability of $u_{1}$ on $[0,y_{1}]$ implies
$$  
1+ \frac{b(y_{1})+H(y_{1})-\lambda}{G'(y_{1})} <1\;,
$$
we finally get
$$
\lambda<b(y_{1})+H(y_{1})\;.
$$
\end{proof}


\begin{thebibliography}{99}

\bibitem{CJL} F.Campillo, M.Joannides, I.Larramendy. Stochastic models
  of the chemostat.  arXiv:1011.5108.



\bibitem{quatres}P.Collet, S.M\'el\'eard, S.Mart\'\i nez, J.San
  Mart\'\i n. Quasi-stationary distributions for structured birth and
 death processes with mutations.  Probab. Theory Related Fields, Volume 151, Issue 1 
(2011), Page 191-231.


\bibitem{CW}K.Crump and W.Young. Some stochastic features of bacterial constant
growth apparatus. Bulletin of Mathematical Biology, \textbf{41} (1979), 53--66.
 
\bibitem{bgm}P. Masci, O. Bernard, F. Grognard, "Continuous selection
  of the fastest growing species in the chemostat", 17th IFAC World
  Congress, Seoul, Korea, 2008.
  
\bibitem{bbgm} P. Masci, F. Grognard, E. Benoit, O. Bernard, "Competitive exclusion in 
the generalized Droop model with N species," submitted to Mathematical Biosciences and 
Engineering, 2009.


\bibitem{dh} D.Dykhuizen and D.Hartl. Selection in chemostats. Microbiol Rev. 1983 
June; 47(2): 150-168.



\bibitem{fl}Flegr,J. (1997): Two distinct types of natural selection
  in turbidostat-like and chemostat-like ecosystems.
  J. theor. Biol. 188: 121-126.


\bibitem{lls} P. De Leenheer, B. Li, and H.L. Smith, Competition in
  the chemostat: Some remarks, Can. Appl. Math. Quar. 11. (2003),
  pp. 229-248

\bibitem{monod} Monod, J.(1950) La technique de culture continue. Ann. Inst. Pasteur 
79 : 390 -410

\bibitem{ns} A.Novick, L.Szilard. Experiments with the Chemostat on
  Spontaneous Mutations of Bacteria. PNAS \textbf{36}, 708-719 (1950).

\bibitem{ns2}Novick A, Szilard L (1950). "Description of the
  Chemostat". Science 112 (2920): 715-6


\bibitem{SW} H.Smith, P.Waltman. \textsl{The Theory of the
  Chemostat}. Cambridge University Press, 1995.
  
  \bibitem{vandoorn}
  E. Van Doorn. Quasi-stationary distributions and convergence to quasi-stationarity of birth-death processes. Adv. in Appl. Probab. 23 no.4, 683--700 (1991).


\end{thebibliography}
\end{document}